\documentclass{article}

\usepackage{theorem}
\usepackage{amsmath,amssymb,amsfonts,graphicx,float}

\def\BBox{\rule{2mm}{3mm}}
\def\QED{\hfill$\BBox$}
\newenvironment{proof}
{\begin{rm}\par\smallskip\noindent{\bf Proof.}\quad}{\QED\end{rm}}

\theorembodyfont{\upshape}

\newtheorem{thm}{Theorem}[section]
\newtheorem{remark}[thm]{\bfseries Remark}    %%   are numbered consecutively

\newtheorem{prop}[thm]{\bfseries Proposition} %%

\theorembodyfont{\rmfamily}
\newtheorem{defn2}[thm]{\bfseries Definition}

\theorembodyfont{\rmfamily}
\newtheorem{defn}[thm]{\bfseries Definition}

\newtheorem{alg}[thm]{\bfseries Algorithm}

\makeatletter
 
 \@addtoreset{equation}{section}
\makeatother

\begin{document}

%\begin{center}

%%%%%%%%%%%%%%%%%%%%%%%%%%%%%%%%%%%%%%%%%%%%%%%%%%%%%%%%
% Title
%%%%%%%%%%%%%%%%%%%%%%%%%%%%%%%%%%%%%%%%%%%%%%%%%%%%%%%%
\title{Complete enumeration of small realizable oriented matroids} 

%%%%%%%%%%%%%%%%%%%%%%%%%%%%%%%%%%%%%%%%%%%%%%%%%%%%%%%%
% begin : Authors
%%%%%%%%%%%%%%%%%%%%%%%%%%%%%%%%%%%%%%%%%%%%%%%%%%%%%%%%

\author{
Komei Fukuda
\thanks{
Research partially supported by the Swiss National Science Foundation Project
	No.~200021-124752/1.}\\
        Institute for Operations Research and \\
        Institute of Theoretical Computer Science, \\
ETH Zurich, Switzerland\\
{\tt fukuda@ifor.math.ethz.ch}
\and
Hiroyuki Miyata
\thanks{Research partially supported by Grant-in-Aid for Scientific Research from the Japan Society for the Promotion of Science, 
and Grant-in-Aid for JSPS Fellows.}\\
Graduate School of Information \\
Sciences,\\
Tohoku University, Japan\\
{\tt hmiyata@dais.is.tohoku.ac.jp}
\and
Sonoko Moriyama
\thanks{Research partially supported by Grant-in-Aid for Scientific
    Research from Ministry of Education, Science and Culture, Japan.}\\
    Graduate School of Information Sciences,\\
     Tohoku University, Japan\\
     {\tt moriso@dais.is.tohoku.ac.jp}
}

%\date{\empty}

\maketitle

\begin{abstract}
Enumeration of all combinatorial types of point configurations and polytopes
is a fundamental problem in combinatorial geometry. 
Although many studies have been done, most of them
are for $2$-dimensional and non-degenerate cases.  

Finschi and Fukuda (2001) published the first database of oriented 
matroids including degenerate (i.e., non-uniform) ones and of higher ranks.
In this paper, we investigate algorithmic ways to classify them 
in terms of realizability, although the underlying decision problem
of realizability checking is NP-hard.  As an application, we determine
all possible combinatorial types (including degenerate ones) of $3$-dimensional configurations of $8$ points, $2$-dimensional configurations of $9$ points and $5$-dimensional configurations of $9$ points.
We could also determine all possible combinatorial types of $5$-polytopes with $9$ vertices.
\end{abstract}

\section{Introduction}
Point configurations and convex polytopes play central roles in
computational geometry and discrete geometry.
For many problems, their {\it combinatorial structures} or 
{\it types} is often more important than their metric structures.
The combinatorial type of a point configuration is defined by all possible partitions of the points by a hyperplane 
(the definition given in (\ref{def:covectors})), 
and encodes various important information such as
the convexity, the face lattice of the convex hull and all possible triangulations.
One of the most significant merits to consider combinatorial types of them is that they are finite for any {\it fixed  sizes\/} (dimension and number of elements)
while there are infinitely many realizations of a fixed combinatorial type. 
This enables us to enumerate those objects and study them through 
computational experiments.
For example, Finschi and Fukuda~\cite{FF03} constructed a counterexample 
for the conjecture by da Silva and Fukuda~\cite{dSF98} using their database~\cite{FF}.
Aichholzer, Aurenhammer and Krasser~\cite{AAK02} and Aichholzer and Krasser~\cite{AK06} developed a database of point configurations~\cite{AAK} 
and showed usefulness of the database
by presenting various applications to computational geometry~\cite{AAK02,AK01,AK06}.

Despite its merits, enumerating combinatorial types of point configurations 
is known to be a quite hard task.
Actually, the recognition problem of combinatorial types of point configurations is
polynomially equivalent to the {\it Existential Theory of the Reals (ETR)},
the problem to decide whether a given polynomial equalities and inequalities system with 
integer coefficients has a real solution or not, even for $2$-dimensional point configurations~\cite{M88,S91}.
Because of this difficulty, enumerations have been done in the following two steps.

The first step is to enumerate a suitable super set of combinatorial types of point configurations 
which can be recognized efficiently.
One of the most frequently used structures has been {\it oriented matroids}.
Oriented matroids are characterized by simple axiom systems and many techniques for the enumeration have been
proposed.
Exploiting a canonical representation of oriented matroids and
algorithmic advances, Finschi and Fukuda~\cite{FF02,FF03} enumerated
oriented matroids including {\it non-uniform} ones (Definition \ref{def:uniform}), degenerate configurations in the abstract setting, 
on up to $10$ elements of rank $3$
and those on up to $8$ elements for every rank.  
In addition, Finschi, Fukuda and Moriyama enumerated {\it uniform} oriented matroids (Definition \ref{def:uniform})
in OM($4,9$) and OM($5,9$) using OMLIB~\cite{FF}, 
where OM($r,n$) denotes the set of all rank $r$ oriented matroids on $n$ elements.
Aichholzer, Aurenhammer and Krasser~\cite{AAK02}, and Aichholzer and Krasser~\cite{AK06} enumerated
uniform oriented matroids on up to $11$ elements of rank $3$.
The enumeration results are summarized in Table \ref{existing1}.
 
In the second step, to obtain all possible combinatorial types of point configurations, 
we need to extract oriented matroids that are {\it acyclic} and {\it realizable}.
Realizable oriented matroids (Definition \ref{def:realizable}) are oriented matroids that can be represented as vector configurations
and acyclic-ness (Definition \ref{def:acyclic}) abstracts the condition that a vector configuration can be associated to a point configuration.
While checking the acyclic-ness is trivial, the realizability problem is polynomially equivalent to ETR~\cite{M88,S91} and thus NP-hard.
In this paper, we show that the realizability problem can be practically
solved for small size instances by exploiting
sufficient conditions of realizability or those of non-realizability.

\subsection{Brief history of related enumeration}
The enumeration of realizable oriented matroids has a long history.
First, Gr\"unbaum~\cite{G03,G72} enumerated all realizable rank $3$ oriented matroids on up to $6$ elements
through the enumeration of hyperplane arrangements.
Then Canham~\cite{C71} and Halsey~\cite{H71} performed enumeration of all realizable rank $3$ oriented matroids on $7$ elements.
Goodman and Pollack~\cite{GP80a,GP80b} proved that rank $3$ oriented matroids on up to $8$ elements are all
realizable.
The enumeration of rank $3$ uniform realizable oriented matroids on $9$ elements is due to Richter~\cite{R88} and Gonzalez-Sprinberg and Laffaille~\cite{GSL89}.
The case of rank $4$ uniform oriented matroids on $8$ elements was resolved by Bokowski and Richter-Gebert~\cite{BRG90}.
Bokowski, Laffille and Richter resolved the case of rank $3$ uniform oriented matroids on $10$ elements (unpublished).
Recently, Aichholzer, Aurenhammer and Krasser~\cite{AAK02} developed a database of all realizable
non-degenerate acyclic oriented matroids of rank $3$ on $10$ elements and then
Aichholzer and Krasser~\cite{AK06} uniform ones of rank $3$ on $11$ elements.
The enumeration results are summarized in Table \ref{existing1_realizable}.

The enumeration of combinatorial types of convex polytopes also has a long history.
The combinatorial types of convex polytopes are defined by {\it face lattices} (See~\cite{G72,Z95}).
All combinatorial types of $3$-polytopes can be enumerated by using Steinitz' theorem~\cite{S22,SR34}.
We can also obtain all combinatorial types of $d$-polytopes with $n$ vertices using Gale diagrams for $n \leq d+3$~\cite{L70,F06}.
On the other hand, the enumeration of combinatorial types of
$d$-polytopes with $(d+4)$ vertices and those of $4$-polytopes
are known to be quite difficult~\cite{M88,R96}.
Gr\"unbaum and Sreedharan~\cite{GS67} listed all combinatorial types of simplical $4$-polytopes with $8$ vertices, and
Altshuler, Bokowski and Steinberg~\cite{ABS80} those of simplicial $4$-polytopes with $9$ vertices, and then
Altshuler and Steinberg~\cite{AS85} those of non-simplicial $4$-polytopes with $8$ vertices.
The enumeration results are summarized in Table \ref{existing2}.
\begin{table}[h]
\begin{center}
\scalebox{0.8}
{
\begin{tabular}{c | c | c | c | c | c | c | c | c | c | c | c |}
        & n = 3 & n = 4 & n = 5 & n = 6 & n = 7 & n = 8 & n = 9 & n = 10 & n = 11\\
 \hline
 r = 3  & \shortstack{1 \\ (1)}     &  \shortstack{2 \\ (1)}    &  \shortstack{4 \\ (1)}    &  \shortstack{17 \\ (4)}    &   \shortstack{143 \\ (11)}   &  \shortstack{4,890 \\ (135)}    &   \shortstack{461,053 \\ (4,382)}   & \shortstack{95,052,532 \\ (312,356)} & \shortstack{unknown \\ (41,848,591)} \\
 \hline 
 r = 4  &       &  \shortstack{1 \\ (1)}    &  \shortstack{3 \\ (1)}    &   \shortstack{12 \\ (4)}   &   \shortstack{206 \\ (11)}  & \shortstack{181,472 \\ (2,628)}   & \shortstack{unknown \\ (9,276,601)} & & \\
 \hline
 r = 5  &       &       &  \shortstack{1 \\ (1)}    &   \shortstack{4 \\ (1)}   &  \shortstack{25 \\ (1)}   & \shortstack{6,029 \\ (135)} & \shortstack{unknown \\ (9,276,601)} & & \\
 \hline
 r = 6  &       &       &        &  \shortstack{1 \\ (1)}   &  \shortstack{5 \\ (1)}    & \shortstack{50 \\ (1)}   & \shortstack{508,321 \\ (4,382)} & & \\
 \hline
 r = 7  &       &       &        &     &  \shortstack{1 \\ (1)}    & \shortstack{6 \\ (1)}   & \shortstack{91 \\ (1)} & \shortstack{unknown \\ (312,356)} & \\
 \hline
 r = 8  &       &       &        &   &      & \shortstack{1 \\ (1)}   & \shortstack{7 \\ (1)} & \shortstack{164 \\ (1)} & \shortstack{unknown \\ (41,848,591)}\\
 \hline
 r = 9  &       &       &        &    &     &   & \shortstack{1 \\ (1)} & \shortstack{8 \\ (1)} & \\
 \hline
 r = 10  &       &       &        &     &      &  &  & & \shortstack{1 \\ (1)} \\
 \hline
\end{tabular}
}
\end{center}
\caption{The numbers of simple oriented matroids ($n$: the number of elements, $r$: rank)
(reorientation class, the numbers enclosed by brackets are those of uniform oriented matroids)
\cite{AAK02,AK06,BRG90,C71,FF02,FF03,GSL89,G03,G72,H71,R88}}
\label{existing1}
\end{table} 

\begin{table}[h]
\begin{center}
\scalebox{0.8}
{
\begin{tabular}{c | c | c | c | c | c | c | c | c | c | c | c |}
        & n = 3 & n = 4 & n = 5 & n = 6 & n = 7 & n = 8 & n = 9 & n = 10 & n = 11\\
 \hline
 r = 3  & \shortstack{1 \\ (1)}     &  \shortstack{2 \\ (1)}    &  \shortstack{4 \\ (1)}    &  \shortstack{17 \\ (4)}    &   \shortstack{143 \\ (11)}   &  \shortstack{4,890 \\ (135)}    &   \shortstack{unknown \\ (4,381)}   & \shortstack{unknown \\ (312,114)} & \shortstack{unknown \\ (41,693,377)} \\
 \hline 
 r = 4  &       &  \shortstack{1 \\ (1)}    &  \shortstack{3 \\ (1)}    &   \shortstack{12 \\ (4)}   &   \shortstack{206 \\ (11)}  & \shortstack{unknown \\ (2,604)}   & \shortstack{unknown \\ (unknown)} & & \\
 \hline
 r = 5  &       &       &  \shortstack{1 \\ (1)}    &   \shortstack{4 \\ (1)}   &  \shortstack{25 \\ (1)}   & \shortstack{6,029 \\ (135)} & \shortstack{unknown \\ (unknown)} & & \\
 \hline
 r = 6  &       &       &        &  \shortstack{1 \\ (1)}   &  \shortstack{5 \\ (1)}    & \shortstack{50 \\ (1)}   & \shortstack{unknown \\ (4,381)} & & \\
 \hline
 r = 7  &       &       &        &     &  \shortstack{1 \\ (1)}    & \shortstack{6 \\ (1)}   & \shortstack{91 \\ (1)} & \shortstack{unknown \\ (312,114)} & \\
 \hline
 r = 8  &       &       &        &   &      & \shortstack{1 \\ (1)}   & \shortstack{7 \\ (1)} & \shortstack{164 \\ (1)} & \shortstack{unknown \\ (41,693,377)}\\
 \hline
 r = 9  &       &       &        &    &     &   & \shortstack{1 \\ (1)} & \shortstack{8 \\ (1)} & \\
 \hline
 r = 10  &       &       &        &     &      &  &  & & \shortstack{1 \\ (1)} \\
 \hline
\end{tabular}
}
\end{center}
\caption{The numbers of simple realizable oriented matroids ($n$: the number of elements, $r$: rank)
(reorientation class, the numbers enclosed by brackets are those of uniform realizable oriented matroids) \cite{AAK02,AK06,BRG90,C71,GSL89,G03,G72,H71,R88}}
\label{existing1_realizable}
\end{table}

\begin{table}[h]
\begin{center}
\scalebox{0.8}
{
\begin{tabular}{c | c | c | c | c | c | c | c | c | c |}
        & n = 3 & n = 4 & n = 5 & n = 6 & n = 7 & n = 8 & n = 9 \\
 \hline
 d = 2  & \shortstack{1 \\ (1)}     &  \shortstack{1 \\ (1)}    &  \shortstack{1 \\ (1)}    &  \shortstack{1 \\ (1)}    &   \shortstack{1 \\ (1)}   &  \shortstack{1 \\ (1)}    &   \shortstack{1 \\ (1)}   \\
 \hline 
 d = 3  &       &  \shortstack{1 \\ (1)}    &  \shortstack{2 \\ (1)}    &   \shortstack{7 \\ (2)}   &   \shortstack{34 \\ (5)}  & \shortstack{257 \\ (14)}   & \shortstack{2606 \\ (50)} \\
 \hline
 d = 4  &       &       &  \shortstack{1 \\ (1)}    &   \shortstack{4 \\ (2)}   &  \shortstack{31 \\ (5)}   & \shortstack{1294 \\ (37)} & \shortstack{unknown \\ (1142)} \\
 \hline
 d = 5  &       &       &        &  \shortstack{1 \\ (1)}   &  \shortstack{6 \\ (2)}    & \shortstack{116 \\ (8)}   & \shortstack{unknown \\ (unknown)} \\
 \hline
\end{tabular}
}
\end{center}
\caption{The numbers of combinatorial types of convex polytopes (the numbers enclosed by brackets are those of simplicial polytopes) ($n$: the number of vertices, $d$: dimension)
\cite{ABS80,AS85,G03,GS67}}
\label{existing2}
\end{table} 
However, there is no database of these objects including degenerate ones or 
of high dimensions, currently.
Many problems in combinatorial geometry remain open especially for high dimensional cases or degenerate cases,
and thus a database of combinatorial types for higher dimensional or degenerate ones will be of great importance.
For example, characterizing the $f$-vectors of $d$-polytopes is a big open problem for $d \geq 4$
while the same questions for 3-polytopes and for simplicial polytopes have already been solved~\cite{S1906,BL81,S80}.

Since Finschi and Fukuda developed a database of oriented matroids~\cite{FF,FF02} containing non-uniform ones,
the realizability classification of larger oriented matroids including non-uniform case has begun.
Existing non-realizability certificates such as non-Euclideanness~\cite{F82,M82} and  biquadratic final polynomials ~\cite{BR90}
and existing realizability certificates such as non-isolated elements~\cite{RS91} and solvability sequence ~\cite{BS86}
were applied to OM($4,8$) and OM($3,9$)~\cite{FMO09,NMFO05,NMF07}.
A new realizability certificate using polynomial optimization and generalized mutation graphs~\cite{NMF07}
and new non-realizability certificates non-HK*~\cite{FMO09} and applying semidefinite programming~\cite{MMI09}
were proposed and applied to OM($4,8$) and OM($3,9$).
Results of those classifications are summarized in Figure~\ref{classification}.
\begin{figure}[ht]
\begin{center}
\includegraphics[scale=0.34]{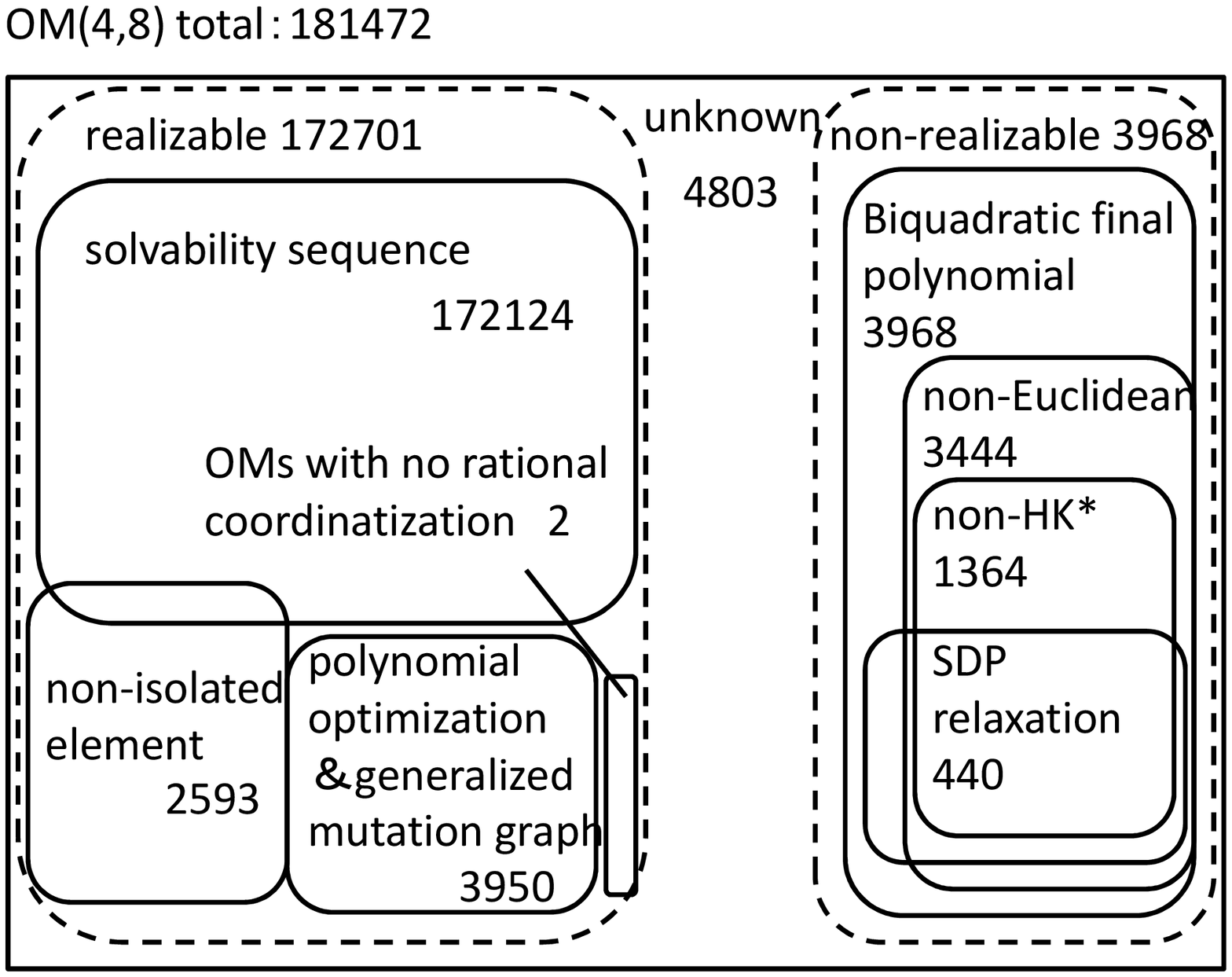}
\includegraphics[scale=0.30]{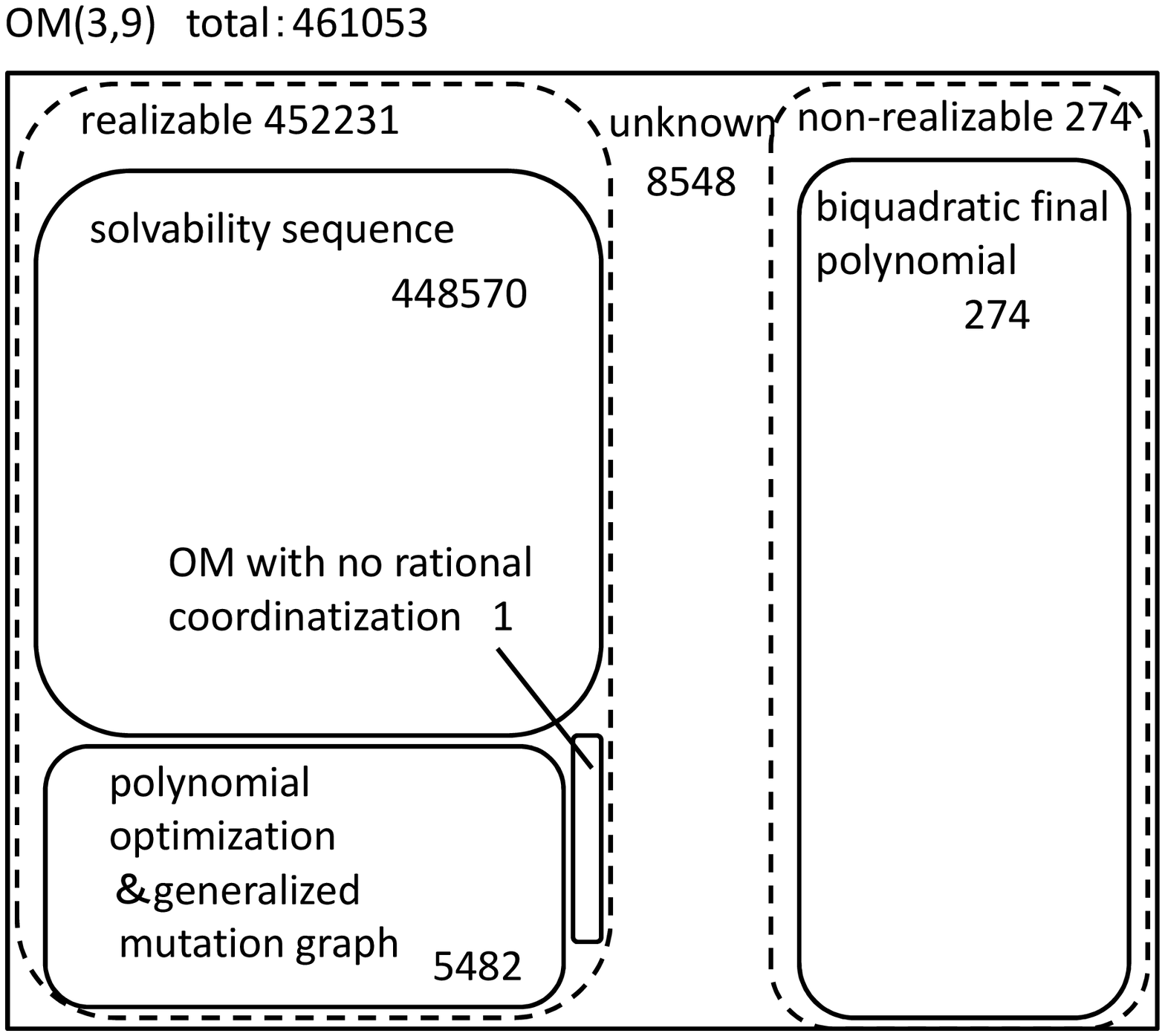}
\end{center}
\caption{Previous classifications of OM($4,8$) \& OM($3,9$) w.r.t. certificates~\cite{FMO09,NMFO05,NMF07,MMI09}}
\label{classification}
\end{figure}

It is important to observe that
there are $4803$ oriented matroids in OM($4,8$) and $8548$
oriented matroids in OM($3,9$) whose realizability was previously unknown.

\subsection{Our contribution}
In this paper, we complete the classification of OM($4,8$), OM($3,9$) and OM($6,9$) (Theorem~\ref{realizability result})
by providing a new method that could successfully find realizations 
of all previously unclassified oriented matroids.

%Recently, it is reported that methods based on random realizations are applied successfully to the realizability problem of
%uniform oriented matroids~\cite{AAK02,AK06} and that of triangulated surfaces~\cite{L08}.
%Those methods work very fast, but they have the follwoing obvious difficulity.
%The probablity that a random realization method generate degenerate configuration is $0$.
%In addition, the probabilty of obtaining a desired configuration is very small for large instances.
%On the other hand, computational algebraic methods such as
%Cylidrical Algebraic Decomposition~\cite{C75} solve the realizability problem precisely, 
%but the problem size  which we can deal with is quite limited.
As mentioned above, the realizability problem is as hard as solving general polynomial inequalities asymptotically.
There are several methods to solve general polynomial inequalities such as Cylindrical Algebraic Decomposition~\cite{C75},
but the problem size which can be practically dealt with 
is quite limited, and our instances turn out to be intractable.
One of the reasons is that those methods compute the complete description 
of cylindrical decomposition of the solution set, which is not necessary for our purpose.
It suffices to find one solution of the polynomial system to decide realizability,
which is usually a much easier task.
It is recently reported that methods based on random realizations are applied successfully to
the classification of the realizability of
uniform oriented matroids~\cite{AAK02,AK06} and that of triangulated surfaces~\cite{L08}.
However, those methods are not directly applicable to non-uniform oriented matroids.

In this paper,  we take a fresh look at the solvability sequence method~\cite{BS86}, which detects the realizability of a given oriented matroid,
provided one can eliminate all variables in the polynomial system using a certain elimination rule.
We extend the elimination rule and introduce some additional techniques so that
they can be applied to a broader class of oriented matroids.  We also
use random realizations when there are remaining variables in the final step.
Using this method, we manage to realize all realizable oriented matroids in OM($4,8$), OM($3,9$) and OM($6,9$).
This in turn proves that every non-realizable oriented matroid in these classes admits a biquadratic final polynomial certificate.

\begin{thm}~\mbox{}
\label{realizability result}
\begin{itemize}
\item[{\rm (a)}] Among $181,472$ oriented matroids in OM($4,8$) (reorientation class), $177,504$ oriented matroids 
are realizable and $3,968$ are non-realizable.
\item[{\rm (b)}] Among $461,053$ oriented matroids in OM($3,9$) (reorientation class), $460,779$ oriented matroids
are realizable and $274$ are non-realizable.
\item[{\rm (c)}] Among $508,321$ oriented matroids in OM($6,9$) (reorientation class), $508,047$ oriented matroids
are realizable and $274$ are non-realizable.
\end{itemize}
\end{thm}
We note here that the classification of OM($6,9$) is obtained from the classification of OM($3,9$) and the duality of oriented matroids~\cite{BLSWZ99}.
As a byproduct, we obtain the following results.
\begin{thm}~\mbox{}
\label{point configurations polytopes}
\begin{itemize}
\item[{\rm (a)}] There are exactly $15,287,993$ combinatorial types of 
$2$-dimensional configurations of $9$ points, \\$105,128,749$ $5$-dimensional configurations of $9$ points and 
$10,559,305$ $3$-dimensional configurations of $8$ points.
\item[{\rm (b)}] There are exactly $47,923$ combinatorial types of $5$-dimensional polytopes with $9$ vertices.
Among them, $322$ are simplicial and $126$ are simplicial neighborly.
\end{itemize}
\end{thm}
Our classification results with certificates are available at\\
{\tt http://www-imai.is.s.u-tokyo.ac.jp/\~{}hmiyata/oriented\_matroids/}\\
To make the results as reliable as possible, 
we recomputed realizability or non-realizability even for oriented matroids 
whose realizability had been known previously.
In the above web page, realizations of all realizable oriented matroids and final polynomials of all non-realizable oriented matroids are uploaded.
One can check correctness of our results there.
\\  
\\
{\bf Organization of the paper:}\\
In Section 2, we present some basic notions on oriented matroids.
Then we discuss a standard method to find realizations in Section 3.
We first review the existing method to decrease the size of a polynomial system.
and explain a new method to search for solutions of polynomial systems.
We apply these methods to the classification of OM($4,8$), OM($3,9$) 
and OM($6,9$) in Section 4,
and conclude the paper in Section 5.
%\begin{cor}~\mbox{}
%\begin{itemize}
%\item Among 177,476 realizable oriented matroids in OM($4,8$), precisely 2 oriented matroids
%are  irrational.
%\item Among 460,779 realizable oriented matroids in OM($3,9$), precisely 1 oriented matroid
%is irrational.
%\item Among 508,047 realizable oriented matroids in OM($6,9$), precisely 1 oriented matroid
%is irrational.
%\item All of combinatorial types of 5-dimensional polytopes on 9 vertices can be realized as rational polytopes.
%\end{itemize}
%\end{cor}

\section{Preliminaries}
In this section, we review basic notions on oriented matroids that are
used in the paper. 
%However, we explain only the chirotope axioms and will not explain the other axioms such as the covector axioms and the circuit axioms here.
For further details about oriented matroids, see~\cite{BLSWZ99}.
\subsection{Point configurations and their combinatorial abstractions}
Let $P=\{ p_{1},\dots,p_{n} \}$ be a point configuration in $\mathbb{R}^{r-1}$.
We define a map $\chi : \{ 1,\dots,n\}^{r} \rightarrow \{ +,-,0 \}$ by
\[ \chi (i_{1},\dots,i_{r}) := {\rm sign}(\det (v_{i_{1}},\dots,v_{i_{r}})), \]
where $v_{1}:=\begin{pmatrix}p_{1} \\ 1\end{pmatrix}$,\dots,$v_{n}:=\begin{pmatrix}p_{n} \\ 1\end{pmatrix} \in \mathbb{R}^{r}$ are
the associated vectors of $p_{1},\dots,p_{n}$.
We define the map $\chi$ as the {\it combinatorial type\/} of $P$, which
satisfies the following properties (a), (b) and (c) with $E= \{ 1,\dots,n\}$.
%For example, in 2-dimensional case, $\chi (i_{1},i_{2},i_{3}) = + \ (-,0)$ implies that points 
%$p_{i_{1}},p_{i_{2}},p_{i_{3}}$
%are located couterclockwisely (clockwisely, colinearly).
\begin{defn2}(Chirotope axioms)\\
Let $E$ be a finite set and $r\geq 1$  an integer. A {\it chirotope of rank $r$ on $E$} is a mapping 
$\chi : E^{r} \rightarrow \{ +1,-1,0\}$ which satisfies the following properties
for any $i_{1},\dots,i_{r},j_{1},\dots,j_{r} \in E$. 
\begin{itemize}
\item[(a)] $\chi$ is not identically zero.
\item[(b)] $\chi (i_{\sigma (1)},\dots,i_{\sigma (r)}) = {\rm sgn} (\sigma) \chi (i_{1},\dots,i_{r})$
for all $i_1,\dots,i_r \in E$ and any permutation $\sigma$.
\item[(c)] If $\chi (j_{s},i_{2},\dots,i_{r})\cdot \chi (j_{1},\dots,j_{s-1},i_{1},j_{s+1},\dots,j_{r}) \geq 0 \text{ for all } s=1,\dots,r, \text{ then }$\\
$\chi (i_{1},\dots,i_{r}) \cdot \chi (j_{1},\dots,j_{r}) \geq 0.$
\end{itemize}
\label{chirotope_axioms}
\end{defn2} 
We note here the third property is an abstraction of Grassmann-Pl\"ucker relations:
\[ [i_{1}\dots i_{r}][j_{1}\dots j_{r}] - \sum_{s=1}^{r}{[j_{s}i_{2}\dots i_{r}][j_{1}\dots j_{s-1}i_{1}j_{s+1}\dots j_{r}] }= 0,\]
where $[i_{1},\dots,i_{r}]:=\det(v_{i_{1}},\dots,v_{i_{r}})$ for all $i_{1},\dots,i_{r} \in E$.
We define an {\it oriented matroid\/} as a pair of a finite set $E$ and a chirotope $\chi:E^{r} \rightarrow \{ +1,-1,0 \}$
satisfying the above axioms.
From now on, we set $E:= \{ 1,\dots,n\}$ throughout this section.
Since $\chi$ is completely determined by the values on 
$\Lambda (n,r):=\{ (i_1,...,i_r) \in \mathbb{N}^r \mid 1 \leq i_1 < \dots < i_r \leq n \}$,
we sometimes regard the restriction $\chi |_{\Lambda (n,r)}$ as the chirotope itself.
We call the pair $(E, \{ \chi ,-\chi \} )$ a {\it rank $r$ oriented matroid with $n(=|E|)$ elements}.
The concept of degeneracy is also defined for oriented matroids as follows.
\begin{defn2}
\label{def:uniform}
An oriented matroid $(E,\{ \chi,-\chi \})$ is said to be {\it uniform\/} 
if $\chi (i_{1},\dots,i_{r}) \neq 0$ for all $(i_{1},\dots ,i_{r}) \in \Lambda (n,r)$,
otherwise {\it non-uniform}.  
\end{defn2}
Note that every ($d+1$)-subset of a $d$-dimensional point configuration spans a $d$-dimensional space if and only if
the underlying oriented matroid is uniform.

For a point configuration $P=\{ p_1,\dots,p_n \} \subset \mathbb{R}^d$, 
the data of chirotope is known to be equivalent to the following data.
\begin{equation}
 {\cal V}_P^* := \{ ({\rm sign}(a^T p_1 - b), {\rm sign}(a^T p_2 - b), \dots, {\rm sign}(a^T p_n - b)) \mid a \in \mathbb{R}^d, b \in \mathbb{R} \}. 
 \label{def:covectors}
\end{equation}
An element of ${\cal V}^*_P$ is called a {\it covector}.
Another axiom system of oriented matroids ({\it Covector axiom}) can be obtained by abstracting properties of ${\cal V}_P^*$.
See \cite[Chapter 3]{BLSWZ99}, for details.

Note that every point configuration has a covector $(++\dots +)$.
Abstracting this property, {\it acyclic} oriented matroids are defined.
%%%%%%%%%%%
\begin{defn2}
\label{def:acyclic}
An oriented matroid is said to be {\it acyclic} if it has a covector ${(++ \dots +)}$.
\end{defn2}
%%%%%%%%%%%
It is known that there is the one-to-one correspondence between acyclic {\it realizable} oriented matroids
and combinatorial types of point configurations.

From ${\cal V}_P^*$, we can read off the convexity of $P$.
For $i=1,\dots,n$, $p_i$ is an extreme point of $P$ if and only if there is a covector  
$\underset{\text{$i$-th component}}{(+\dots +0+\dots +)} \in {\cal V}_P^*$.
In this way, {\it matroid polytopes} are introduced as abstractions of convex point configurations.
\begin{defn2}
\label{def:matroid_polytope}
An acyclic oriented matroid on a ground set $\{ 1,\dots,n\}$ 
is called a {\it matroid polytope} if it has a covector 
$\underset{\text{$i$-th component}}{(+\dots +0+\dots +)}$ for all $i=1,\dots,n$.
\end{defn2}
For a matroid poltyope, its facets are defined by non-negative {\it cocircuits} i.e., non-negative minimal covectors.
For details on matroid polytopes, see~\cite[Chapter 9]{BLSWZ99}.

\subsection{The realizability problem}
Every vector configuration has the underlying oriented matroid, 
but the converse is not true because ``non-realizable'' oriented matroids exist.
%A problem to decide whether a given oriented matoroid has the corresponding vector configuration or not is
%called the realizability problem of orieted matroids. 
\begin{defn2}(The realizability problem of oriented matroids)\label{realizability_def}\\
Given a rank $r$ oriented matroid $M=(E,\{ \chi ,-\chi \} )$ with $n$ elements,
the {\it realizability problem for\/} $(E,\{ \chi ,-\chi \} )$ is  to decide
whether the following polynomial system has a real solution $v_{1},\dots,v_{n} \in \mathbb{R}^{r}$:
\begin{equation}
\label{realizability}
 {\rm sign}(\det(v_{i_1},\dots,v_{i_r})) = \chi (i_1,\dots,i_r) 
 \text{ for all }  (i_1,...,i_r) \in \Lambda (n,r).
 \end{equation}
\end{defn2}

\begin{defn2}
\label{def:realizable}
A rank $r$ oriented matroid is said to be {\it realizable} if it arises from an $r$-dimensional vector configuration,
otherwise {\it non-realizable}.
\end{defn2}
Not every realizable oriented matroid admits a point configuration
because a positive combination of some vectors can be $0$
while the same is not true for the associated vector configuration of any point configuration.
We can extract combinatorial types of point configurations by picking up acyclic realizable oriented matroids.

\subsection{Isomorphic classes}
In this paper, we consider only {\it simple oriented matroids}, those without parallel elements and loops, see~\cite{BLSWZ99}.
For simple oriented matroids, we consider the following two equivalent classes.
\begin{defn2} Let $M=(E,\{ \chi, -\chi \} )$ and $M'=(E,\{ \chi', -\chi' \} )$ be oriented matroids.
\begin{itemize}
\item[(a)] $M$ and $M'$ are {\it relabeling equivalent\/} if
\[ \chi(i_{1},\dots,i_{r})=\chi '(\phi (i_{1}),\dots,\phi (i_{r})) \text{ for all $i_1,\dots,i_r \in E$} \]
  or
\[ \chi(i_{1},\dots,i_{r})=- \chi '(\phi (i_{1}),\dots,\phi (i_{r})) \text{ for all $i_1,\dots,i_r \in E$}\]
for some permutation $\phi$ on $\{ 1,\dots,n\}$.
\item[(b)] $M$ and $M'$ are {\it reorientation equivalent\/} if $M$ and $-_{A}M'$ are relabeling equivalent for some $A \subset E$,
where $-_{A}M'$ is the oriented matroid determined by the chirotope $-_{A}\chi'$ defined as follows. 
\[ -_{A}\chi' (i_1,\dots,i_r):=(-1)^{|A \cap \{ i_1,\dots,i_r\} |}\chi'(i_1,\dots,i_r) \text{ for
$i_1,\dots,i_r \in E$.} \]
\end{itemize}
\end{defn2}
We note here that the realizability is completely determined by the {\it reorientation classes} of oriented matroids, the equivalence
classes defined by the reorientation equivalence. 
This is because  if any oriented matroid in a reorientation class
is realizable, then every oriented matroid in the class is realizable. 
A database of oriented matroids by Finschi and Fukuda~\cite{FF} consists of the representatives of the reorientation classes.
On the other hand, we say that point configurations $P,P'$ has the same combinatorial type 
if the combinatorial type of $P$ and that of $P'$ belong to the same relabeling class.

\section{Methods to recognize realizability of oriented matroids}
Recognizing that a given oriented matroid is realizable amounts to
 finding a solution of the associated polynomial system (\ref{realizability})
in Definition \ref{realizability_def}.

Our strategy is as follows.
We  first simplify the polynomial system as much as possible, namely,
eliminate as many variables as possible using simple elimination rules
and then try random realizations if no further elimination of variables
is possible by the elimination rules.

\subsection{Inequality reduction techniques}
There are three critical parameters to measure difficulty of solving polynomial systems:
the number of variables, the degrees of variables and the number of constraints.
We must try to keep each of them low. For such purposes, we employ
some techniques used in \cite{BRS90,BS86, N07,NMF07}.
Let us review these techniques briefly.

\subsubsection{Fixing a  basis} \label{BasisFix}
A technique explained in the following was introduced in~\cite{BS86} 
and was used to reduce the degrees of constraints and the number of variables in~\cite{BS86,N07,NMF07}.

We assume that $\chi (i_1,...,i_r)=+$ for some $i_1,...,i_r$ by taking the negative of $\chi$ if necessary.
Let $M_V:=(v_1,...,v_n) \in \mathbb{R}^{d \times n}$ be the representation matrix of a vector configuration $V$.
Because the combinatorial type of $V$ is invariant under any invertible linear transformations, we can assume that
$v_{b_{1}}=(1,0,...,0)^{T},v_{b_{2}}=(0,1,0,...,0)^{T},...,v_{b_{r}}=(0,...,0,1)^{T}$ for an $r$-tuple $(b_1,...,b_r) \in \Lambda (n,r)$
such that $\chi (b_1,...,b_r) = +$.
We call such an $r$-tuple of indices a {\it basis}.
We obtain a new polynomial system as follows.
\begin{equation} \label{BasisFix:eq1}
\begin{cases}
{\rm sign}(\det(v_{i_{1}},...,v_{i_{r}})) = \chi (i_{1},...,i_{r}) \text{ for all $(i_{1},...,i_{r}) \in \Lambda(n,r),$ and} \\
v_{b_{1}}=(1,0,...,0)^{T},v_{b_{2}}=(0,1,0,...,0)^{T},...,v_{b_{r}}=(0,...,0,1)^{T}.
\end{cases}
\end{equation}
The resulting polynomial system (\ref{BasisFix:eq1}) depends on the choice of bases.
In the next section, we present how to find a suitable basis.

Finally, note that the degree of each constraint can be computed easily by the following formula:
\[ deg(\det(v_{i_{1}},...,v_{i_{r}})))= |\{ b_{1},....,b_{r} \} \setminus \{ i_{1},...,i_{r} \} |, \]
and that the sign of each variable $v_{kl}$ is determined by the obvious equation:
\[ v_{kl}=\det(v_{b_{1}},...,v_{b_{k-1}},v_{l},v_{b_{k+1}},...,v_{b_{r}}).\]

\subsubsection{Minimal reduced systems} \label{MinReduced}
Most of the techniques in this section are
introduced in~\cite{BS86} for uniform oriented matroids
to reduce the number of constraints
and are extended to general oriented matroids in~\cite{N07,NMF07}.

Recall that in the polynomial system (\ref{BasisFix:eq1}), which describes the condition of realizability, there are $\binom{n}{r}$ constraints.  
There are some possible redundancies in these constraints, as one may reconstruct $\chi$ according to partial values of $\chi$ by 
using Axiom (c) of a chirotope.
For example, if $\chi (1,2,3)=\chi (1,4,5)=\chi (1,2,4)=\chi (1,3,5) = \chi (1,2,5)=+$,
then we obtain $\chi (1,3,4)=+$ using Axiom (c).
For a subset $R$ of $\Lambda (n,r)$, we denote $\langle R \rangle$
the set of all $r$-tuples whose $\chi$ signs are implied by the sign information
on $\chi$ over $r$-tuples in $R$ and the chirotope axioms.   
A minimal subset $R(\chi )$ of $\Lambda (n,r)$ needed to decide $\chi$ is called a {\it minimal reduced system\/}
for $(E,\{ \chi , - \chi \} )$.
First, we observe that {\it generalized mutations} to be defined below
must be contained in every minimal reduced system of $\chi$.
\begin{defn2}
An $r$-tuple $\lambda \in \Lambda (n,r)$ is called a {\it generalized mutation\/} of $\chi$ if there exists an oriented matroid $(E ,\{ \chi',-\chi'\})$
such that
$\chi (\mu ) = \chi'(\mu )$ for all $\mu \in \Lambda (n,r) \setminus \{ \lambda \}$
and
$\chi (\lambda ) \neq \chi'(\lambda )$.
\end{defn2}
This definition is different from that of~\cite{NMF07,N07}, but turns out 
to be more natural because of the proposition below.

Let us denote the set of all generalized mutations of $\chi$ by $GMut(\chi )$.
A nice characterization of the mutations of uniform oriented matroids is given in~\cite[Theorem 3.4.]{RS88},
but such a characterization is not known for generalized mutations.
However, one can compute $GMut(\chi )$ using the following proposition, which is an immediate extension of ~\cite[Proposition 3.3.]{RS88}.
\begin{prop}
An $r$-tuple $\lambda :=(i_1,\dots,i_r) \in \Lambda (n,r)$ is a generalized mutation of $\chi$
if and only if $\lambda$ {\it is not determined by Grassmann-Pl\"ucker relations},
i.e., the following condition holds:
\begin{equation}
\label{Gmut_Grassmann}
\begin{cases}
\chi (j_{1},\dots, j_{r}) = 0, 
\text{ or } \\
\chi (j_{1},\dots, j_{r}) \neq 0 \text{ and }
\{ \chi (j_{s},i_{2},\dots, i_{r})\chi (j_{1},\dots, j_{s-1},i_{1},j_{s+1},\dots, j_{r}) \mid  s=1,...,r\} \supset \{ +,- \}
\end{cases}
\end{equation}
for all $1 \leq j_1 < \dots < j_r \leq n$.
\begin{proof} 
In this proof, we use the notation $\langle k_1,\dots,k_r \rangle$ to represent an $r$-tuple $(k'_1,\dots,k'_r)$ such that 
$\{ k'_1,\dots,k'_r\} = \{ k_1,\dots,k_r\}$
and $k'_1 \leq \dots \leq k'_r$, for $k_1,\dots,k_r \in \mathbb{N}$.

To prove the if-part, let us take a map $\chi:\Lambda (n,r) \rightarrow \{ +,-,0\}$  satisfying Condition (\ref{Gmut_Grassmann}).
We consider a map $\chi':\{1,\dots,n\} \rightarrow \{ +,-,0\}$ satisfying Axioms (a) and (b) of a chirotope
and $\chi' (\mu) = \chi (\mu)$ for all $\mu \in \Lambda (n,r) \setminus \{ \lambda \}$ and $\chi (\lambda ) \neq \chi' (\lambda )$.
We prove that $\chi'$ also satisfies Axiom (c).
For $(k_1,\dots,k_r), (l_1,\dots,l_r) \in \Lambda (n,r)$, assume that
\begin{equation}
 \chi' (l_{s},k_{2},\dots,k_{r})\cdot \chi' (l_{1},\dots,l_{s-1},k_{1},l_{s+1},\dots,l_{r}) \geq 0 \text{ for all } s=1,\dots,r. \label{assumption}
\end{equation} 
Under this condition, we prove $\chi' (k_1,\dots,k_r)\chi' (l_1,\dots,l_r) \geq 0$ by the following case analyses.
\\
\\
{\bf (I)} $k_1 \in \{ l_1,\dots,l_r\}$.\\
Let $t \in \{ 1,\dots,r\}$ be an integer such that $k_1 = l_t$.
\[ \chi' (l_{t},k_{2},\dots,k_{r})\cdot \chi' (l_{1},\dots,l_{t-1},k_{1},l_{t+1},\dots,l_{r}) 
= \chi' (k_1,\dots,k_{r})\cdot \chi' (l_{1},\dots,l_{r}) \geq 0. \]
{\bf (II)} $k_1 \notin \{ l_1,\dots,l_r\}$.\\
{\bf (II-A)} $|\{ k_1,\dots,k_r\} \cap \{ l_1,\dots,l_r \} | = r-1$.\\
Let us take $l_u \notin \{ k_1,\dots,k_r \}$. Then we have
\begin{align*}
\chi' (l_{u},k_{2},\dots,k_{r}) \cdot \chi' (l_{1},\dots,l_{u-1},k_{1},l_{u+1},\dots,l_{r}) 
&=  
(-1)^u \cdot \chi' (l_{1},\dots,l_{r}) \cdot (-1)^u \cdot \chi' (k_{1},\dots,k_{r}) \\
&= \chi' (l_{1},\dots,l_{r}) \cdot \chi' (k_{1},\dots,k_{r}) \geq 0. 
\end{align*}
{\bf (II-B)} $|\{ k_1,\dots,k_r\} \cap \{ l_1,\dots,l_r\}| < r-1$.\\
{\bf (i)} $\langle k_1,\dots,k_r \rangle = \lambda$.\\
Since $k_1 \notin \{ l_1,\dots,l_r\}$ and $|\{ k_1,\dots,k_r\} \cap \{ l_1,\dots,l_r\}| < r-1$,
$\langle l_{s},k_{2},\dots,k_{r} \rangle \neq \lambda$ and $\langle l_{1},\dots,l_{s-1},k_{1},l_{s+1},\dots,l_{r} \rangle \neq \lambda$
for $s=1,\dots,r$.
Therefore, Condition (\ref{assumption}) implies
\[ \chi (l_{s},k_{2},\dots,k_{r})\cdot \chi (l_{1},\dots,l_{s-1},k_{1},l_{s+1},\dots,l_{r}) \geq 0 \text{ for all } s=1,\dots,r.\]
To satisfy Condition (\ref{Gmut_Grassmann}), the following must hold.
\[ \chi'(l_1,\dots,l_r) =  \chi (l_1,\dots,l_r) = 0. \]
{\bf (ii)} $\langle l_1,\dots,l_r \rangle = \lambda$.\\
Proved similarly to Case (i).
\\
\\
{\bf (iii)} $\langle k_1,\dots,k_r \rangle \neq \lambda$, $\langle l_1,\dots,l_r \rangle \neq \lambda$.\\
First, we consider the case when there exists $s_0$ such that $ \lambda = \langle l_{s_0},k_2,\dots,k_{r} \rangle$
and $\chi (l_1,\dots,l_{s_0-1},k_1,l_{s_0+1},\dots,l_r) \neq 0$.
Let us write down Condition (\ref{Gmut_Grassmann}) for $\chi$ under $i_1:=l_{s_0},i_2:=k_2,\dots,i_r:=k_r$ 
and $j_1:=l_1,\dots,j_{s_0-1}:=l_{s_0-1},j_{s_0}:=k_1,j_{s_0+1}:=k_{s_0+1},\dots,j_r:=l_r$:
\begin{align*} 
& \{ \chi (l_{s},k_{2},\dots, k_{r})\chi (l_{1},\dots, l_{s_0-1},k_{1},l_{s_0+1},\dots, l_{s-1},l_{s_0},l_{s+1},\dots,l_r) \mid  s=1,...,s_0-1\} \\
& \cup
\{ \chi (l_{s},k_{2},\dots, k_{r})\chi (l_{1},\dots, l_{s-1},l_{s_0},l_{s+1},\dots, l_{s_0-1},k_1,l_{s_0+1},\dots,l_r) \mid  s=s_0+1,...,r\} \\
&\cup
\{ \chi (k_{1},k_{2},\dots, k_{r})\chi (j_{1},\dots, j_{r}) \mid  s=1,...,r\} \\
&= \\
& \{ (-1) \cdot \chi (l_{s},k_{2},\dots, k_{r})\chi (l_{1},\dots, l_{s_0-1},l_{s_0},l_{s_0+1}\dots l_{s-1},k_1,l_{s+1},\dots,l_r) \mid  s=1,...,s_0-1\} \\
& \cup
\{ (-1) \cdot \chi (l_{s},k_{2},\dots, k_{r})\chi (l_{1},\dots, l_{s-1},k_1,l_{s+1}\dots l_{s_0-1},l_{s_0},l_{s_0+1},\dots,l_r) \mid  s=s_0+1,...,r\} \\
&\cup
\{ \chi' (k_{1},\dots, k_{r})\chi' (l_{1},\dots, l_{r}) \} \\
&\supset \{ +,- \} 
\end{align*}
This implies $\chi' (k_{1},\dots, k_{r})\chi' (l_{1},\dots, l_{r}) = +$.
\\
We consider the other case. In this case, Condition (\ref{assumption}) implies
\[ \chi (l_{s},k_{2},\dots,k_{r})\cdot \chi (l_{1},\dots,l_{s-1},k_{1},l_{s+1},\dots,l_{r}) \geq 0 \text{ for all } s=1,\dots,r\]
and thus
\[ \chi' (k_{1},\dots, k_{r})\chi' (l_{1},\dots, l_{r}) = \chi (k_{1},\dots, k_{r})\chi (l_{1},\dots, l_{r}) = +.  \]
Finally, we conclude $\chi' (k_{1},\dots, k_{r})\chi' (l_{1},\dots, l_{r}) \geq 0$ for all cases.
This proves the if-part. 

The only if-part is proved by contradiction.
Suppose that there exist $1 \leq j_1 < \dots < j_r \leq n$ such that
\[
\chi (j_{1},\dots, j_{r}) \neq 0 \text{ and }
\{ \chi (j_{s},i_{2},\dots, i_{r})\chi (j_{1},\dots, j_{s-1},i_{1},j_{s+1},\dots, j_{r}) \mid  s=1,...,r\} = \text{ $\{ + \}$ (or $\{- \}$ or $\{ 0\}$).}
\]
In this case, we obtain $\chi (i_1,\dots,i_r) = \chi(j_1,\dots,j_r) \chi (j_{1},i_{2},\dots, i_{r}) \chi (i_{1},j_{2},\dots, j_{r})$ using 
Axiom (c) of a chirotope.
On the other hand, for a map $\chi':\Lambda(n,r) \rightarrow \{ +,-,0 \}$ taking the same value as that of $\chi$ except for $\lambda$,
the following holds.
\begin{align*}
& \chi' (i_1,\dots,i_r) 
\neq
\chi (i_1,\dots,i_r) 
=\chi (j_1,\dots,j_r) \chi (j_{1},i_{2},\dots, i_{r}) \chi (i_{1},j_{2},\dots, j_{r})\\
&= \chi'(j_1,\dots,j_r) \chi' (j_{1},i_{2},\dots, i_{r}) \chi' (i_{1},j_{2},\dots, j_{r}). 
\end{align*}
This means that $\chi'$ violates Axiom (c) of a chirotope.
\end{proof}
\end{prop}

The above proposition shows that $GMut(\chi ) \subset R(\chi )$.
%We note here that considering $3$-term Grassmann-Pl\"ucker relations is not sufficient in the non-uniform case
%and thus considered all Grassmann-Pl\"ucker relations.
%Starting from $GMut(\chi )$,
Therefore, we can compute $R(\chi )$ by the following procedure, which is a slightly modified version of
the algorithm in~\cite{BRS90,N07}.
\begin{alg}(Computing a minimal reduced system for $\chi$)\\
Input: A chirotope $\chi :\{ 1,...,n \}^{r} \rightarrow \{ +,-,0 \}$,  $b \in \Lambda (n,r)$ s.t. $\chi (b) \neq 0$.\\
Output: A minimal reduced system $R(\chi )$.\\
{\bf Step 1:} Compute $GMut(\chi )$ and set $R:=GMut(\chi )$.\\
{\bf Step 2:} $C:= \langle R \rangle$. If $C =\Lambda (n,r)$, return $R$. Otherwise go to Step $3$.\\
{\bf Step 3:} Choose $\mu \in \Lambda (n,r) \setminus C$ that minimizes $|\mu \setminus b|$, and add it to $R$. Go to Step $2$.
\end{alg}

As a result, we obtain a new polynomial system as follows.
\begin{equation} \label{MinReduced:eq1}
 {\rm sign}(\det(v_{i_{1}},...,v_{i_{r}})) = \chi (i_{1},...,i_{r}) \text{ for $(i_{1},...,i_{r}) \in R(\chi ).$}
\end{equation}
One might be able to simplify the reduced polynomial system 
by selecting a different basis.
One way is to search a better basis $b'$, in the sense of the totality of the degrees of the constraints,
by computing $\sum_{\beta \in R(\chi )}{|b' \setminus \beta |}$.
Various weight functions of constraints are considered in~\cite{BRS90,N07}.
This is the subject of the next section.

\subsubsection{Normalization (eliminating homogeneity)} \label{Normalization}
A technique to be explained here is used in~\cite{N07,NMF07} to reduce the number of variables and the degrees
of polynomials, and in addition, to eliminate some of the equality constraints.

First, we negate negative variables to make all variable non-negative,
and obtain the following new polynomial system:
\begin{align*}
\begin{cases}
{\rm sign}(\det(v'_{i_{1}},...,v'_{i_{r}})) = (-1)^{s_{i_{1}...i_{r}}} \cdot \chi (i_{1},...,i_{r}) \text{ for $(i_{1},...,i_{r}) \in R(\chi ).$} \\
v'_{b_{1}}=(1,0,...,0)^{T},v'_{b_{2}}=(0,1,0,...,0)^{T},...,v'_{b_{r}}=(0,...,0,1)^{T}, 
\end{cases}
\end{align*}
where $s_{i_{1}...i_{r}}$ denotes the number of negative variables in $v_{1i_{1}},...,v_{ri_{1}},v_{1i_{2}},...,v_{ri_{r}}$.

Let $\alpha_{1},...,\alpha_{n},\beta_{1},...,\beta_{r}$ be arbitrary positive numbers and
$w_{1},...,w_{r} \in \mathbb{R}^{n}$ be row vectors of $M_{V'}:=(v'_{1},...,v'_{n})$.
Then
\[ \det(\alpha_{i_{1}}v'_{i_{1}},...,\alpha_{i_{r}}v'_{i_{r}})=\alpha_{i_{1}} \cdots \alpha_{i_{r}}\det(v'_{i_{1}},...,v'_{i_{r}}), \]

\[ \det
\begin{pmatrix}
\beta_{1}w_{1}(i_{1},...,i_{r})\\
\vdots \\
\beta_{r}w_{r}(i_{1},...,i_{r})
\end{pmatrix}
=
\beta_{1}\cdots \beta_{r}
\det
\begin{pmatrix}
w_{1}(i_{1},...,i_{r})\\
\vdots \\
w_{r}(i_{1},...,i_{r})
\end{pmatrix},\]
where $w_{j}(i_{1},...,i_{r})$ denotes the vector whose $k$-th element is an $i_k$-th element of $w_{j}$ for $1 \leq i_1 < ... < i_r \leq n$
and $1 \leq j \leq r$.
This allows us to
fix two indices $1 \leq i \leq n$ and $1 \leq j \leq r$ 
and to assume that every component of the vectors $v'_{i}$ and $w_{j}$ is $0$ or $1$.

Furthermore, we can eliminate some of equality constraints as follows.
Let us classify the constraints arising from $\chi (i_1,...,i_r)$ according to the values of $|\{ b_1,...,b_r \} \setminus \{ i_1,...,i_r \} |$.
\begin{description}
\item[Case (a)] $|\{ b_1,...,b_r \} \setminus \{ i_1,...,i_r \} | = 1$. \\
The constraint corresponds to the sign constraint ``$v'_{ij} > 0$'' or ``$v'_{ij} = 0$'' for certain $1 \leq i \leq n$, $1 \leq j \leq r$.
\item[Case (b)] $|\{ b_1,...,b_r \} \setminus \{ i_1,...,i_r \} | = 2$. \\
Let $\{ b_1,...,b_r \} \setminus \{ i_1,...,i_r \} =: \{ i,i'\}$ and $\{ i_1,...,i_r \}  \setminus \{ b_1,...,b_r \} =: \{ j,j'\}$.
Then the constraint is of the form $v'_{ij}v'_{i'j'}-v'_{i'j}v'_{ij'}=0$, where each variable may be fixed to $0$ or $1$.
If the row $i$ is normalized, the constraint become $v'_{i'j'}-v'_{i'j}=0$ and we can eliminate the constraint in the trivial way
without increasing the degree of the polynomial system.
If the column $j$ is normalized, the constraint become $v'_{ij'}-v'_{i'j'}=0$ and can be eliminated in the trivial way.
\item[Case (c)] $|\{ b_1,...,b_r \} \setminus \{ i_1,...,i_r \} | > 2$.\\
The constraint cannot be eliminated trivially. 
\end{description}

\begin{remark}
Equality constraints are main obstacles for random realizations.
The normalization technique turns out very useful in
removing equations so that random realizations can be
applied.
In addition, the technique decreases the number of variables and that of constraints by $1$.
Therefore, we minimize the number of equality constraints over all choices of bases and columns and rows to be normalized,
which can be computed easily.
\end{remark}

%KF

\subsection{Searching for solutions of polynomial systems} \label{SSPS}
Now we present a practical method to find a solution of a polynomial system.

Cylindrical Algebraic Decomposition method~\cite{C75}
 eliminates variables preserving the feasibility until 
polynomial systems contain only one variable, 
solves the one-variable polynomial systems,
and then lifts the projected solutions.
It may sound quite simple but each step is a very hard task in general.
Therefore, we take a fresh look at the solvability sequence method~\cite{BS86},
by which some variables might be eliminated in a simpler way.
\begin{prop}\label{elimination1} {\rm (\cite{BS86})}
Let $l_{1},l_{2},l_3 \geq 0$ be integers and 
$R^{(1)}_i,R^{(2)}_i,L^{(1)}_j,L^{(2)}_j,P_k$ be polynomials for $i=1,...,l_1, j=1,...,l_2, k=1,...,l_3$.
Then the  feasibility of rational polynomial system:
\begin{equation} \label{SSPS:eq1}
\begin{cases}
y < \frac{R^{(1)}_{i}(x_{1},...,x_{n})}{R^{(2)}_{i}(x_{1},...,x_{n})} &  (i=1,...,l_{1}),\\
y > \frac{L^{(1)}_{j}(x_{1},...,x_{n})}{L^{(2)}_{j}(x_{1},...,x_{n})} &  (j=1,...,l_{2}),\\
P_k(x_1,...,x_n) >  (\text{or }=) \ 0 & (k=1,...,l_3)\\ 
\end{cases}
\end{equation}
is equivalent to that of the following polynomial system:
\begin{equation} \label{SSPS:eq2}
\begin{cases}
L^{(1)}_{j}(x_{1},...,x_{n})R^{(2)}_{i}(x_{1},...,x_{n}) < R^{(1)}_{i}(x_{1},...,x_{n})L^{(2)}_{j}(x_{1},...,x_{n}) & (i=1,...,l_{1},j=1,...,l_{2}), \\
P_k(x_1,...,x_n) >  (\text{or }=)  \ 0 & (k=1,...,l_3)
\end{cases}
\end{equation}
under the condition $R^{(2)}_{i}(x_{1},...,x_{n})L^{(2)}_{j}(x_{1},...,x_{n}) > 0$ for $i=1,...,l_{1},j=1,...,l_{2}$.
\end{prop}
Note that a solution $(y^{*},x_{1}^{*},...,x_{n}^{*})$ of the original system 
(\ref{SSPS:eq1})
can be constructed from a solution $(x_{1}^{*},...,x_{n}^{*})$ of the resulting system by 
\[y^{*}:= \frac{\min{\{ R_{i}(x_{1}^{*},...,x_{n}^{*}) \mid i=1,...,l_{1} \} } + \max{\{ L_{j}(x_{1}^{*},...,x_{n}^{*}) \mid j=1,...,l_{2} \} }}{2}.\]
This elimination rule is used in the solvability sequence method under the {\it bipartiteness condition} for determinant systems, i.e.,
polynomial systems whose constraints are of the form ``${\rm sign}(\det(v_{i_{1}},...,v_{i_{r}})) = \chi (i_{1},...,i_{r})$.''
In the determinant system, one can detect the signs of $R^{(2)}_{i}(x_{1},...,x_{n})$ and $L^{(2)}_{j}(x_{1},...,x_{n})$
in advance using the information of $\chi$, and can solve the polynomial system by $y$.
\begin{defn}~\mbox{(\cite{BS86})}\\
Consider the polynomial system (\ref{SSPS:eq1}) arising from a determinant system.
Then each constraint can be rewritten as ``${\rm sign}(\det(v_{i_{1}},...,v_{i_{r}})) = \chi (i_{1},...,i_{r})$'' 
where $(i_1,\dots,i_r) \in \Lambda (n,r)$ and can be indexed by $\{ i_1,\dots,i_r \}$.
Let $A$ be the set of indices which defines constraints of the form $y < \frac{R^{(1)}_{i}(x_{1},...,x_{n})}{R^{(2)}_{i}(x_{1},...,x_{n})}$
and $B$ the index set whose elements define constraints of the form $y > \frac{L^{(1)}_{j}(x_{1},...,x_{n})}{L^{(2)}_{j}(x_{1},...,x_{n})}$.
The polynomial system (\ref{SSPS:eq1}) is said to be bipartite if
$|\{ i_1,\dots,i_r \} \cap \{ j_1,\dots,j_r\}| = r-1$ for all $\{ i_1,\dots,i_r \} \in A, \{ j_1,\dots,j_r\} \in B$.
\end{defn}
In \cite{BS86}, it is proved that the feasibility of the polynomial system (\ref{SSPS:eq1}) is equivalent to that of the following polynomial system 
under the bipartiteness condition:
\[ P_k(x_1,...,x_n) >  (\text{or }=) \ 0 \ (k=1,...,l_3). \]
Therefore, the elimination does not produce new constraints under the bipartiteness condition, and one can proceed with eliminations without creating inconsistency.
It detects realizability if all variables are eliminated by this elimination rule.
However, the restriction of the bipartiteness condition is very strong.
We can easily get rid of the restriction of bipartiteness condition by allowing
the elimination rules to {\it destroy} the determinant system.
This relaxation turns out to be very useful when it is used together with {\it branching rules},
which will be explained later.

Before explaining branching rules, we also consider an elimination rule for polynomial systems containing equalities.
\begin{prop} \label{elimination2}
Let $l_{1},l_{2},l_3 \geq 0$ be integers and
$P_i,E_j,E,Q_k$ be rational polynomials for $i=1,...,l_1, j=1,...,l_2, k=1,...,l_3$.
Then the  feasibility of rational polynomial system:
\begin{equation} \label{SSPS:eq3}
 \begin{cases}
P_{i}(x_{1},...,x_{n},y) > 0 & (i=1,...,l_{1}), \\
E_{j}(x_{1},...,x_{n},y) = 0 & (j=1,...,l_{2}), \\
y = E(x_1,...,x_n), \\
Q_k(x_1,...,x_n) > (\text{or }=) \ 0 & (k=1,...,l_3) \\ 
\end{cases} 
\end{equation}
is equivalent to that of the following rational polynomial system:
\begin{equation} \label{SSPS:eq4}
\begin{cases}
P_{i}(x_{1},...,x_{n},E(x_1,...,x_n)) > 0 & (i=1,...,l_{1}), \\
E_{j}(x_{1},...,x_{n},E(x_1,...,x_n)) = 0 & (j=1,...,l_{2}), \\
Q_k(x_1,...,x_n) > (\text{or }=)  \ 0 & (k=1,...,l_3)
\end{cases}
\end{equation}
\end{prop}
\begin{proof}
A solution $(x_{1}^{*},...,x_{n}^{*},y^*)$ of the original system
(\ref{SSPS:eq3}) can be constructed 
from a solution $(x_{1}^{*},...,x_{n}^{*})$ of the resulting system (\ref{SSPS:eq4}) by
$y^{*}:=E(x_{1}^{*},...,x_{n}^{*})$. 
\end{proof}
\\
\\
In Propositions \ref{elimination1} and \ref{elimination2}, we regard a variable $y$ appearing in these forms
as {\it redundant}.
%We try to eliminate variables of this type as far as possible.
To apply these elimination rules to as many variables of this type as possible, we consider
the following additional rules, which are called {\it  branching rules}.
\begin{prop}\label{branching1}
Let $l_1,l_2 \geq 0$.
The polynomial system:
\begin{equation}\label{branching1_ps1}
\begin{cases}
 A_{i}(x_{1},...,x_{n})y < B_{i}(x_{1},...,x_{n}) & (i=1,...,l_1), \\
 P_j(x_1,...,x_n) > (\text{or }=)  \ 0 & (j=1,...,l_2)
 \end{cases}
 \end{equation}
is feasible if and only if one of the following rational polynomial systems is feasible
\begin{equation} \label{branching1_ps2}
 \begin{cases}
{\rm sign}(A_{i}(x_{1},...,x_{n})) = s(i) & (i=1,...,l_1), \\
y < \frac{B_{i}(x_{1},...,x_{n})}{A_{i}(x_{1},...,x_{n})} & (i \leq l_1,s(i)=+),\\
y > \frac{B_{i}(x_{1},...,x_{n})}{A_{i}(x_{1},...,x_{n})} & (i \leq l_1,s(i)=-),\\
B_i(x_1,...,x_n) > 0 & (i \leq l_1,s(i)=0),\\
P_j(x_1,...,x_n) > (\text{or }=)  \ 0 & (j=1,...,l_2)
\end{cases} 
\end{equation}
for $s:\{ 1,...,l\} \rightarrow \{ +,-,0\}$.
\end{prop}
\begin{proof}
We partition $\mathbb{R}^{n+1}$ into $3^{l_1}$ parts $R_s:=\{ (x_1,...,x_n,y) \mid {\rm sign}(A_{i}(x_{1},...,x_{n})) = s(i) \ (i=1,...,l_1) \}$
for $s:\{ 1,...,l\} \rightarrow \{ +,-,0\}$.
Let $S$ be the solution space of the polynomial system (\ref{branching1_ps1}).
Then $S \cap R_s$ is represented by the polynomial system (\ref{branching1_ps2}).
$S$ is non-empty if and only if $S \cap R_s$ is non-empty for some $s:\{ 1,...,l\} \rightarrow \{ +,-,0\}$.
\end{proof}
\\
\\
Observe that one can actually reduce {\it the search range} as follows
\begin{equation} \label{branching1_ps3}
\begin{cases}
{\rm sign}(A_{i}(x_{1},...,x_{n})) = s(i) & (i=1,...,l_1), \\
y < \frac{B_{i}(x_{1},...,x_{n})}{A_{i}(x_{1},...,x_{n})} & (i \leq l_1,s(i)=+),\\
y > \frac{B_{i}(x_{1},...,x_{n})}{A_{i}(x_{1},...,x_{n})} & (i \leq l_1,s(i)=-),\\
P_j(x_1,...,x_n) > (\text{or }=)  \ 0 & (j=1,...,l_2)
\end{cases} 
\end{equation}
for $s:\{ 1,...,l_1\} \rightarrow \{ +,-\}$.
\\
\\
This is because $(x^*_1+\epsilon_1,\dots,x^*_n+\epsilon_n,y^*+\epsilon_{n+1})$ is a solution of
(\ref{branching1_ps1}) for a solution $(x^*_1,\dots,x^*_n,y^*)$ of (\ref{branching1_ps1}) and
any sufficiently small $\epsilon_1,\dots,\epsilon_n > 0$.
The following is a branching rule for the case when $y$ appears in equations.
\begin{prop}\label{branching2}
Let $l_1,l_2,l_3 \geq 0$.
The polynomial system:
\begin{equation} \label{branching2_ps1}
\begin{cases}
A_{i}(x_1,...,x_n,y) > 0 & (i=1,...,l_1), \\ 
A_{j}(x_1,...,x_n,y) = 0 & (j=l_1+1,...,l_1+l_2), \\
A(x_1,...,x_n)y = B(x_1,...,x_n), \\
P_k(x_1,...,x_n) > (\text{or }=)  \ 0 & (k=1,...,l_3)
\end{cases} 
\end{equation}
is feasible if and only if one of the following rational polynomial systems is feasible
\begin{align}\label{branching2_ps2}
& \notag \begin{cases}
A(x_1,...,x_n) = 0, & \\
B(x_1,...,x_n) = 0, & \\
A_{i}(x_{1},...,x_{n},y) > 0  & (i=1,...,l_1),\\
A_{j}(x_1,...,x_n,y) = 0 & (j=l_1+1,...,l_1+l_2),\\
P_k(x_1,...,x_n) > (\text{or }=)  0 & (k=1,...,l_3) 
\end{cases}
\\
&\begin{cases}
A(x_1,...,x_n) > 0, & \\
A_{i}(x_1,...,x_n,\frac{B(x_1,...,x_n)}{A(x_1,...,x_n)}) > 0 & (i=1,...,l_1),\\
A_{j}(x_1,...,x_n,\frac{B(x_1,...,x_n)}{A(x_1,...,x_n)}) = 0 & (j=l_1+1,...,l_1+l_2),\\
P_k(x_1,...,x_n) > (\text{or }=)  \ 0 & (k=1,...,l_3)\\
\end{cases}
\\
& \notag \begin{cases}
A(x_1,...,x_n) < 0, & \\
A_{i}(x_1,...,x_n,\frac{B(x_1,...,x_n)}{A(x_1,...,x_n)}) > 0 & (i=1,...,l_1),\\
A_{j}(x_1,...,x_n,\frac{B(x_1,...,x_n)}{A(x_1,...,x_n)}) = 0 & (j=l_1+1,...,l_1+l_2),\\
P_k(x_1,...,x_n) > (\text{or }=)  \ 0 & (k=1,...,l_3)
\end{cases}
\end{align}
\end{prop}
\begin{proof}
This is proved similarly to Proposition \ref{branching1}. We consider the following partition of $R^{n+1}$:
$R_s:=\{(x_1,...,x_n,y) \mid {\rm sign}(A(x_1,...,x_n)) = s\}$ for $s=+,-,0$.
\end{proof}
\\
\\
%We choose a square-free variable to eliminate, apply the branching rule and the elimination rule, and then
%search {\it a feasible branch}.
To solve the polynomial system, the following operations are applied repeatedly.
We first choose a variable $y$ that can be eliminated by the above 4 rules.
Then the branching rule in Proposition \ref{branching1} or Proposition \ref{branching2} is applied to obtain a set of polynomial system
(\ref{branching1_ps2}) or (\ref{branching2_ps2}).
We choose a sign pattern and apply the elimination rule in Proposition \ref{elimination1} or Proposition \ref{elimination2}, 
and check the feasibility recursively.
If feasibility of the polynomial system for this sign pattern is certified, the original polynomial system
is proved to be feasible. Otherwise, we backtrack and try another sign pattern until feasibility is certified.
If no feasible polynomial system is found, we give up deciding feasibility.
We adopt the following termination condition.
If all variables are eliminated and the system is consistent, 
the original system is feasible.
If a system does not have variables that can be eliminated, we try 
random realizations to prove the feasibility.
There are many polynomial systems, which are trivially feasible but are hard to solve algebraically.
For example, the following polynomial system is clearly feasible but is algebraically complicated.
\[ x^{100}-y^{49}+1000y^{23} < 1, \  x^5 > 2, \ y^5 > 2. \]
Random realizations work well as long as the solution set is sufficiently large
and are not affected so much by the algebraic complexity.

\begin{alg}
Sol($P$) ($P$: polynomial system)
\begin{enumerate}
\item If there are no variables that can be eliminated by the above $4$ rules in $P$, try random assignments to the remaining variables.
If a solution is found, return `feasible.' Otherwise return `unknown.'
\item Choose a variable $y$ to eliminate in $P$. Apply one of the branching rules (Proposition \ref{branching1}, Proposition \ref{branching2})
to obtain a set of polynomial systems $P'_1,\dots,P'_m$.
\item For $i = 1,\dots,m$, apply one of the elimination rules (Proposition \ref{elimination1}, Proposition \ref{elimination2}) 
to $P'_i$ and obtain a new polynomial system $Q_i$. If Sol($Q_i$) returns `feasible,' return `feasible.'
\item Return `unknown.'
\end{enumerate}
\end{alg}

%{\bf Remark:}
%The resulting systems are quite different according to the elimination orders of variables, and 
%whether the method works well or not is depends  heavily on the elimination orders of variables.
%On the other hand, normalization of variables and fixing a basis never {\it disturbs} variable eliminations.

%It is a much easier task to find a solution of the polynomial system  because it suffices to find a solution
%at some node.
%If we eliminate all variables at some node and obtain a consistent system, we prove the feasibility.
%In addition, we would also like to prove the feasibility of the polynomial systems as follows.
%\[ x^2-2xy+y^{2} > 0. \]
%\[ a^{2}-b^{3} > 0, a^{3} - 2b^{2} < 0. \]
%The above systems are {\it clearly} feasible, but there exist no efficient and unified algorithm to prove the feasibility known to the authors.
%We propose to use random assignments to variables in order to prove the feasibility.
We sometimes generate too many branchings and
thus apply the iterative lengthy search to the following tree search problem.
The root node consists of the original polynomial system.
Starting from the root node, we expand nodes using the elimination rules and the branching rules repeatedly.
We decide whether we arrive at goal nodes or not using random assignments to the remaining variables.
In this setting, we define the cost of each node $x$ by $c(x):=log_{2}(n_b)$, where $n_b$ is the number of branching
at $x$, and apply the iterative lengthening search to it by increasing the limit of the total cost by $1$.
\begin{remark}
Equalities containing no square-free variables cannot be eliminated by the above rules.
The probability of yielding a solution to equality constraints by random realizations
is $0$ and if we cannot eliminate all equality constraints, it is highly unlikely 
that the above method generate a realization.
In this case, which is quite rare for small instances,
we need to use general algorithms such as
Cylindrical Algebraic Decomposition~\cite{C75} to find solutions.

We point out another tractable case, namely, the case when the ideal generated by the equality constraints 
is zero-dimensional (as an ideal in $\mathbb{C}[x_s \mid s \in S]$, which denote the polynomial ring over $\mathbb{C}$
in variables $x_s$ which appear in equality constraints).
In other words, it is the case when the equality constraints have a finite number of complex solutions.
In this case, we extract real solutions, substitute each solution to the original polynomial system and apply the above methods.
We can check the zero-dimensionality and solve the equalities using Gr\"obner basis~\cite{CLO92}.
\end{remark}

\section{Realizability classification} \label{RealClass}
We apply our method to OM($4,8$) and OM($3,9$). All computations are made on a cluster of 4 servers, with each node having
total 16-core CPUs (each core running at 2.2 GHz) and 128 GB RAM (each process uses only 1 CPU).

First, we apply the polynomial reduction techniques described in section 3.1 to  OM($3,9$).
Table \ref{variables} and Table \ref{constraints} show the distributions of the number of variables and constraints of the resulting polynomial systems
for OM($3,9$).

\begin{table}[h]
\begin{center}
\begin{tabular}{| c | c | c | c | c | c | c | c | c | c | c | c |  }
\hline
No. of variables & 1& 2 & 3 & 4 &  5 & 6 & 7 & 8 & 9 & 10 & 11-18 \\
 \hline
No. of OMs & 0 &  21 & 1,411 & 13,261 & 47,888 & 91,855 & 121,977 & 107,869 & 59,284 & 16,814 & 673\\
 \hline 
\end{tabular}
\end{center}
\caption{Number of the variables (OM($3,9$))}
\label{variables}

\begin{center}
\begin{tabular}{| c | c | c | c | c | c | c | c | c | c | c | c | c | c |  }
\hline
No. of variables & 1-2 & 3-4 & 5-6 & 7-8 &  9-10 & 11-12 & 13-14 & 15-16 & 17-18 \\
 \hline
No. of OMs & 659 & 22,340 & 107,465 & 168,995 & 114,595 &  38,944 & 7,330 & 694 & 31\\
 \hline 
\end{tabular}
\end{center}
\caption{Number of the constraints (OM($3,9$))}
\label{constraints}
\end{table} 
We note here that the case of more than 10 variables occurs as an exceptional case, where
the polynomial system consists of the following type of constraints.
\[ xy > zw, \ x,y,z,w > 0.\]
Nakayama~\cite{N07} proved that oriented matroids with such polynomial systems which
do not admit  biquadratic final polynomials are realizable. 
We detect polynomial systems of this type and stop the polynomial system reductions
because our method of realizations can solve such polynomial systems easily.

We apply our method to search for realizations to the resulting polynomial systems
%For random assingments, we consider the uniform distribution of $\{ n/100 \mid n=1,2,...,10,000 \}^{m}$, where
%$m$ is the number of variables and we try random assignments $1000$ times at every node.
%Our method manages to find realizations to
and manage to find realizations of $460,778$ oriented matroids in OM($3,9$).
Table \ref{time} shows the distribution of the time which our method consumed to find realizations.

\begin{table}[h]
\begin{center}
\begin{tabular}{| c | c | c | c | c | c | c | c | c | c |   }
\hline
Time (ms) & $1$-$10$ & $10$-$10^2$ & $10^2$-$10^3$ & $10^3$-$10^4$ & $10^4$-$10^5$ & $10^5$-$10^6$ & $10^6$-$10^7$ & $10^7$-$10^8$ & $10^8$-$10^9$ \\
 \hline
No. of OMs & $0$ & $241,593$ & $163,927$ & $25,893$ & $19,559$ & $7,534$ & $1,599$ & $615$ & $58$ \\
 \hline 
\end{tabular}
\end{center}
\caption{Computation time (OM($3,9$))}
\label{time}
\end{table} 

Remaining one oriented matroid turns out to be an irrational one, which was found by Perles (See~\cite[p.73]{G03}).
We can realize it using Gr\"obner basis.

As a result, we obtain complete classification of OM($3,9$).
In addition, it leads classification of OM($6,9$) because
the realizability of oriented matroids are preserved by the duality~\cite{BLSWZ99}.
Similarly, we apply our method to OM($4,8$) and manage to give realizations to all realizable oriented matroids in OM($4,8$) (Theorem~\ref{realizability result})
except for two irrational ones found by Nakayama~\cite{N07}.
These two oriented matroids can also be realized using Gr\"obner basis.
\begin{thm}
There are precisely $1,1$ and $2$ irrational realizable oriented matroids in OM($3,9$), OM($6,9$) and OM($4,8$) respectively.
\end{thm}

From these results, we obtain the combinatorial types of point configurations (Theorem~\ref{point configurations polytopes})
by generating relabeling classes of acyclic realizable oriented matroids.
Matroid polytopes (i.e., acyclic oriented matroids with all elements {\it extreme points}. For details, see \cite{BLSWZ99}.) are extracted from them. 
Then we compute the face lattices of the matroid polytopes and decide whether they occur from
some realizable matroid polytopes or not in order to obtain combinatorial types of polytopes  (Theorem~\ref{point configurations polytopes}).
All face lattices of matroid polytopes in OM($4,8$),OM($3,9$) and OM($6,9$) turned out to be realizable as those of convex polytopes.
Table~\ref{no_acyclic} and Table \ref{no_face_lattices} summarize the results.
\\
%\\
%\\
%{\bf Results:}
%\begin{itemize}
%\item[{\rm (a)}]
%Among $10,775,236$ relabeling classes of acyclic oriented matroids in OM($4,8$),
%$10,559,305$ are realizable and $215,931$ are non-realizable.\\
%Among $250,601$ relabeling classes of matroid polytopes in OM($4,8$),
%$238,399$ are realizable and $12,202$ are non-realizable.
%There are $257$ non-isomorphic face lattices of the matroid polytopes and all of them are realizable as the face lattice of polytopes.
%\item[{\rm (b)}]
%Among $15,296,266$ relabeling classes of acyclic oriented matroids in OM($3,9$),
%$15,287,993$ are realizable and $8,273$ are non-realizable.\\
%There are exactly $1$ relabeling class of matroid polytopes in OM($6,9$) and it is realizable.
%
%\item[{\rm (c)}]
%Among $105,183,749$ relabeling classes of acyclic oriented matroids in OM($6,9$),
%$105,128,749$ are realizable and $55,000$ are non-realizable.\\
%Among $41,030,709$ relabeling classes of matroid polytopes in OM($6,9$),
%$41,008,968$ are realizable and $21,741$ are non-realizable.
%There are $47,923$ non-isomorphic face lattices of the matroid polytopes and all of them are realizable as the face lattice of polytopes.
%\end{itemize}
%
\begin{table}[h]
\begin{center}
\begin{tabular}{| c | c | c | }
\hline
          & acyclic OMs (realizable/non-realizable) & matroid polytopes (realizable/non-realizable) \\
 \hline
OM($4,8$) & $10,775,236$ ($10,559,305/215,931$) & $250,601$ ($238,399/12,202$) \\ 
 \hline 
 OM($3,9$) & $15,296,266$ ($15,287,993/8,273$) & $1$ ($1/0$) \\
 \hline
 OM($6,9$) & $105,183,749$ ($105,128,749/55,000$) & $41,030,709$ ($41,008,968/21,741$) \\
 \hline
\end{tabular}
\end{center}
\caption{Numbers of relabeling classes of acyclic OMs and matroid polytopes}
\label{no_acyclic}
\end{table} 
\begin{table}[h]
\begin{center}
\begin{tabular}{| c | c |  }
\hline
          & total (realizable/non-realizable)\\
 \hline
OM($4,8$) & $257$ ($257/0$) \\
\hline
 OM($3,9$) & $1$ ($1/0$) \\
 \hline
 OM($6,9$) & $47,923$ ($47,923/0$) \\
 \hline
\end{tabular}
\end{center}
\caption{Numbers of non-isomorphic face lattices of matroid polytopes}
\label{no_face_lattices}
\end{table} 
\\
The number of non-isomorphic face lattices of $3$-polytopes with $8$ vertices coincides with the number in \cite[p.424, Table 2]{G03}.
We observe that one can associate a point configuration with rational coordinates
to the combinatorial type of every $5$-polytope with $9$ vertices and thus obtain the following theorem.
\begin{thm}
The combinatorial type of every $5$-polytope with $9$ vertices can be realized by a rational polytope.
\end{thm}

\section{Concluding Remarks}
In this paper, we complete the realizability classification of OM($4,8$), OM($3,9$) and OM($6,9$)
by developing new techniques to search for a realization of a given oriented matroid.
Surprisingly, the biquadratic final polynomial method~\cite{BR90}, which is based on a linear programming relaxation,
can detect all non-realizable oriented matroids in these classes.
In addition, one can also find all non-realizable uniform oriented matroids in OM($3,10)$ and OM($3,11$)
by this method~\cite{AAK02,AK06}.
A known minimal non-realizable oriented matroid which cannot be 
determined to be non-realizable by the method
is in OM($3,14$)~\cite{R96}.
It may be of interest to find a minimal example with such property.

Our classification almost reaches the limit of today's computational environments.
However, we can deal with larger instances if the number of instances is small.
Actually, it was successfully applied to OM($4,9$) and OM($5,9$) in order to find the hyperplane arrangements
maximizing the average diameters~\cite{DMMX10} and PLCP-orientations on $4$-cube~\cite{FKM}.
Our classification results are available at\\
{\tt http://www-imai.is.s.u-tokyo.ac.jp/\~{}hmiyata/oriented\_matroids/}

\section*{Acknowledgement}
The authors would like to thank Dr. Hiroki Nakayama for providing his programs.
The computations throughout the paper were done on a cluster server of ERATO-SORST
Quantum Computation and Information Project, Japan Science and Technology Agency.


\begin{thebibliography}{99}
\bibitem{AAK}O. Aichholzer, Order Types for Small Point Sets,\\
{\tt http://www.ist.tugraz.at/staff/aichholzer/research/rp/triangulations/ordertypes/}
\bibitem{AAK02}O. Aichholzer, F. Aurenhammer, and H. Krasser,
 Enumerating order types for small point sets with applications, {\it Order}, 19:265-281, 2002.
\bibitem{AK01} O. Aichholzer and H. Krasser,
 The point set order type data base: A collection of applications and results,
  {\it In Proceedings of the 13th Canadian Conference on Computational Geometry (CCCG 2001)}, pp. 17-20, 2001.
\bibitem{AK06}O. Aichholzer and H. Krasser, Abstract order type extension and new results on the rectilinear crossing number,
 {\it Computational Geometry: Theory and Applications}, 36:2-15, 2006.
\bibitem{ABS80} A. Altshuler, J. Bokowski and L. Steinberg,
The classification of simplicial 3-spheres with nine vertices into polytopes and nonpolytopes. {\it Discrete Mathematics} 31:115-124, 1980.
\bibitem{AS85} A. Altshuler and L. Steinberg,
The complete enumeration of the $4$-polytopes and $3$-spheres with eight vertices.
{\it Pacific Journal of Mathematics}, 117:1-16, 1985.
\bibitem{BL81} L. J. Billera and C. W. Lee, A proof of the sufficiency of
 McMullen's conditions for f-vectors of simplicial convex polytopes, 
{\it Journal of Combinatorial Theory Series A} 31:237-255, 1981.
\bibitem{BLSWZ99} A. Bj\"orner, M. Las Vergnas B. Sturmfels, N. White and G. Ziegler, {\it Oriented Matroids},
Cambridge University Press, 1999.
\bibitem{BR90} J. Bokowski and J. Richter, On the finding of final polynomials. {\it European Journal
of Combinatorics}, 11:21-34, 1990.
\bibitem{BRG90} J. Bokowski and J. Richter-Gebert, On the classification of
non-realizable oriented matroids, Preprint Nr. 1345, Technische
Hochschule, Darmstadt.
\bibitem{BRS90}J. Bokowski, J. Richter, and B. Sturmfels, Nonrealizability proofs in computational geometry,
 {\it Discrete and Computational Geometry}, 5:333-350, 1990.
\bibitem{BS86} J. Bokowski and B. Sturmfels, On the coordinatization of oriented matroids,
{\it Discrete and Computational Geometry}, 1:293-306, 1986.
\bibitem{C71}
R. J. Canham, Ph. D. Thesis, University of East Anglia, 1971.
\bibitem{C75}G. E. Collins, Quantifier elimination for real closed fields by
cylindrical algebraic decomposition, {\it Springer Lecture Notes in
Computer Science}, 33:515.532, 1975.
\bibitem{CLO92} D. Cox, J. Little, and D. O'Shea, {\it Ideals, varieties, and algorithms}, 
 Springer-Verlag, 1992.
\bibitem{dSF98}I. P. F. da Silva and K. Fukuda, Isolating points by lines in the
plane, {\it Journal of  Geometry}, 62(1-2):48-65, 1998.
\bibitem{DMMX10}A. Deza, H. Miyata, S. Moriyama  and F. Xie,
Hyperplane Arrangements with Large Average Diameter: a Computational Approach,
{\it Advanced Studies in Pure Mathematics}, 62:59-74, 2012. 
\bibitem{F82} J. Edmonds, K. Fukuda, {\it Oriented Matroid Programming}, Ph. D. thesis, University of Waterloo,
1982.
\bibitem{FKM} K. Fukuda, L. Klaus and H. Miyata, work in progress.
\bibitem{FL78} J. Folkman and J. Lawrence, Oriented Matroids, {\it J. Combinatorial Theory, Ser. B}, 25:199-236, 1978.
\bibitem{FF} L. Finschi, Homepage of Oriented Matroids, 
{\tt http://www.om.math.ethz.ch/}
\bibitem{FF02}
L. Finschi and K. Fukuda, Generation of oriented matroids - a graph theoretical
approach. {\it Discrete and Computational Geometry}, 27:117-136, 2002.
\bibitem{FF03}L. Finschi and K. Fukuda, Combinatorial generation of small point configurations and
hyperplane arrangements, {\it Discrete and Compuational
Geometry, The Goodman-Pollack Festschrift, Algorithms and Combinatorics
25 Springer}, 425-440, 2003.
\bibitem{FMO09} K. Fukuda, S. Moriyama, and Y. Okamoto, The Holt-Klee condition for oriented
matroids, {\it European Journal of Combinatorics}, 30:1854-1867, 2009. 
\bibitem{F06} \'E. Fusy, Counting d-polytopes with d+3 vertices, {\it Electronic Journal of Combinatorics, Volume 13, R23}, 2006.    
\bibitem{GSL89} G. Gonzalez-Sprinberg and G. Laffaille, Sur les arrangements
simples de huit droites dans $\mathbb{RP}^2$, {\it C. R. Acad. Sci. Paris S\'er. I
Math.}, 309:341-344, 1989. 
\bibitem{GP80a} J. E. Goodman, R. Pollack, On the combinatorial classification of nondegenerate configurations in the plane,
{\it Journal of Combinatorial Theory, Series A}, 29:220-235, 1980. 
\bibitem{GP80b} J. E. Goodman, R. Pollack, Proof of Gr\"unbaum's conjecture on the stretchability of certain arrangements of pseudolines,
{\it Journal of Combinatorial Theory, Series A}, 29:385-390, 1980.
\bibitem{G72} B. Gr\"unbaum, {\it Arrangements and Spreads}, American Mathematical Society, 1972.
\bibitem{G03}  B. Gr\"unbaum, {\it Convex Polytopes 2nd edition}, Springer, 2003. 
\bibitem{GS67} B. Gr\"unbaum and V. P. Sreedharan, An enumeration of simplicial 4-polytopes with 8 vertices, {\it Journal of Combinatorial Theory} 
2:437-465, 1967. 
\bibitem{H71} E. Halsey, Ph. D. Thesis, University of Washington, 1971.
\bibitem{L70} K. Lloyd, The number of d-polytopes with (d+3) vertices. Mathematika, 17:120-132, 1970.
\bibitem{L08} F. H. Lutz, Enumeration and random realization of triangulated surfaces, 
{\it Discrete Differential Geometry} (A. I. Bobenko, P. Schr\"oder, J. M. Sullivan, and G. M. Ziegler, eds.). Oberwolfach Seminars 38,
 235-253. Birkh\"auser, Basel, 2008.
\bibitem{M82} A. Mandel, {\it Topology of oriented matriods}, Ph. D. Thesis, University of Waterloo, 1982.
\bibitem{MMI09} H. Miyata, S. Moriyama and H. Imai,
Deciding non-realizability of oriented matroids by semidefinite programming,
{\it Pacific Journal of  Optimization}, 5:211-224, 2009.
\bibitem{M88} M. N. Mn\"ev, The universality theorems on the classification problem of
configuration varieties and convex polytopes varieties, {\it Topology and Geometry, volume 1346 of Lecture Notes in Mathematics}, pages 527-544,
Springer,Heidelberg, 1988.
\bibitem{N07} H. Nakayama, {\it Methods for Realizations of Oriented Matroids and Characteristic Oriented Matroids}, 
Ph. D. Thesis, the University of Tokyo, 2007.
\bibitem{NMFO05} H. Nakayama, S. Moriyama, K. Fukuda, and Y. Okamoto,
Comparing the strengths of the non-realizability certificates for oriented matroids,
In {\it Proceedings of 4th Japanese-Hungarian Symposium on Discrete Mathematics and Its Applications},
pages 243-249, 2005.
\bibitem{NMF07} H. Nakayama, S. Moriyama and K. Fukuda, 
Realizations of non-uniform oriented matroids using generalized mutation graphs. 
In {\it Proceedings of 5th Hungarian-Japanese Symposium on Discrete Mathematics and Its Applications}, pp.242-251, 2007.
%\bibitem{NMF} H. Nakayama, S. Moriyama and K. Fukuda,
%Three characteristic rank 4 oriented matroids,
%submitted. 
\bibitem{R88} J. Richter, {\it Kombinatorische Realisierbarkeitskriterien f\"ur orientierte
Matroide}, Master's thesis, Technische Hochschule, Darmstadt,
1988. published in Mitteilungen Mathem. Seminar Giessen,
Heft 194, Giessen 1989.
\bibitem{R96} J. Richter-Gebert, Realization spaces of polytopes, 
{\it volume 1643 of Lecture Notes in Mathematics}, Springer-Verlag, Berlin, 1996.
\bibitem{R96_2} J. Richter-Gebert, Two interesting oriented matroids, {\it Documenta Mathematica} 1:137-148, 1996.
\bibitem{RS91} J. Richter and B. Sturmfels, On the topology and geometric construction
of oriented matroids and convex polytopes, {\it Transactions of the American
Mathematical Society}, 325:389-412, 1991.
\bibitem{RS88}J.-P. Roudneff and B. Sturmfels, Simplicial cells in arrangements and mutations of oriented matroids,
{\it Geometriae Dedicata}, 27:153-170, 1988.
\bibitem{S91}P. W. Shor, Stretchability of pseudolines is {NP}-hard, {\it Applied Geometry and
Discrete Mathemaics The Victor Klee Festschrift (P. Gritzmann, B. Sturmfels,
eds.), DIMACS Series in Discrete Mathematics and Theoretical Computer
Science, Amer. Math. Soc., Providence, RI}, 4, 531-554, 1991.
\bibitem{S80}P. Stanley, The number of faces of simplicial convex
Polytopes, {\it Advances in Math}. 35:236-238, 1980.
\bibitem{S1906}E. Steinitz, Uber die Eulerschen Polyederrelationen, 
{\it Archiv fur Mathematikund Physik}, 11:86-88, 1906.
\bibitem{S22} E. Steinitz, Polyeder und Raumeinteilungen, {\it Encyclop\"adie der mathematischen Wissenschaften, Band 3 (Geometries)},
pp. 1.139, 1922.
\bibitem{SR34}E. Steinitz and H. Rademacher, {\it Vorlesungen \"uber die Threorie der Polyeder}, Berlin, 1934.
\bibitem{Z95} G. M. Ziegler, {\it Lectures on Polytopes}, Springer-Verlag, 1995. 
\end{thebibliography}
\end{document}